\documentclass{gtpart}     
\usepackage{epsfig,color} 
\usepackage{mathtools,amsmath}
\usepackage{tikz}
\usetikzlibrary{decorations.pathmorphing}

\title{Garside theory and subsurfaces:\\  some examples in braid groups}

\author{Saul Schleimer}
\givenname{Saul}
\surname{Schleimer}
\address{Saul Schleimer, Mathematics Institute, University of Warwick, Coventry CV4 7AL, UK}
\email{s.schleimer@warwick.ac.uk}

\author{Bert Wiest}
\givenname{Bert}
\surname{Wiest}
\address{Bert Wiest, Univ Rennes, CNRS, IRMAR - UMR 6625, F-35000 Rennes}
\email{bertold.wiest@univ-rennes1.fr}


\volumenumber{}
\issuenumber{}
\publicationyear{}
\papernumber{}
\startpage{}
\endpage{}
\doi{}
\MR{}
\Zbl{}
\received{}
\revised{}
\accepted{}
\published{}
\publishedonline{}
\proposed{}
\seconded{}
\corresponding{}
\editor{}
\version{}

\makeautorefname{notation}{Notation}%

\newtheorem{theorem}{Theorem}[section]
\newtheorem{lemma}[theorem]{Lemma} 
\newtheorem{proposition}[theorem]{Proposition}

\newtheorem{conjecture}[theorem]{Conjecture}
\newtheorem{observation}[theorem]{Observation}
\newtheorem*{unnumberedtheorem}{Theorem}
\theoremstyle{definition}
\newtheorem{definition}[theorem]{Definition}    
\newtheorem{remark}[theorem]{Remark}

\newtheorem{example}[theorem]{Example}

\newtheorem{question}[theorem]{Question}

\newtheorem{claim}[theorem]{Claim}
\def\B4dual{B_4^{\rm \,dual}}
\renewcommand{\phi}{\varphi}
\def\Bthin{1.4pt} 
\def\Bthick{4pt} 
\def\Bdiag{1.5pt} 
\def\bx{.35} 
\def\by{.5}

\begin{document}

\begin{abstract}
Garside-theoretical solutions to the conjugacy problem in braid groups depend on the determination of a characteristic subset of the conjugacy class of any given braid, e.g.\ the sliding circuit set. It is conjectured that, among rigid braids with a fixed number of strands, the size of this set is bounded by a polynomial in the length of the braids.  In this paper we suggest a more precise bound: for rigid braids with $N$ strands and of Garside length~$L$, the sliding circuit set should have at most $C\cdot L^{N-2}$ elements, for some constant~$C$. We construct a family of braids which realise this potential worst case.
Our example braids suggest that having a large sliding circuit set is a geometric property of braids, as our examples have multiple subsurfaces with large subsurface projection; thus they are ``almost reducible'' in multiple ways, and act on the curve graph with small translation distance.
\end{abstract}

\maketitle


\section{Introduction}

In this paper we will study the \emph{conjugacy search problem} in the $N$-strand braid group~$B_N$ (with $N\geqslant 3$): we are looking at algorithms which take as their input two words (whose letters are elements of some finite generating set of~$B_N$ and their inverses), and whose output is the information whether these words represent conjugate elements in~$B_N$. Moreover, if they do, then the algorithm should find an explicit conjugating element. 

It is currently an open problem to find a polynomial-time algorithm, i.e.\ to find such an algorithm for any number of strands~$N$ such that the running time can be bounded by some polynomial $P_N(L)$ as a function of the length~$L$ of the longer one of the two input words.

Even more ambitious is the quest for a \emph{uniform} polynomial-time solution to the conjugacy problem, which means the following:
we start with some generating set of~$B_{\infty}$ whose intersection with~$B_N$, for any integer~$N$, is a finite generating set of~$B_N$. Now we are looking for 
a polynomial~$P$ and for an algorithm which takes as its input two words in this generating set belonging to~$B_N$, for any~$N$, and decides whether these words represent conjugate elements of~$B_N$. The computation time should be bounded by $P(I)$, where $I$ denotes the bit-size of the input. (Note that $I$ is more than just the sum of the lengths of the input words, since the description of the generators adds to their bit-size.)

There are two families of approaches to these problems. 
\begin{itemize} \item 
Geometric approaches, using the curve complex, subsurface projections and the notion of hierarchically hyperbolic spaces, train track splitting sequences or flip sequences of triangulations, the space of projective measured foliations $\mathcal{PMF}$ etc. These approaches typically talk not only about braid groups, but more generally about mapping class groups of surfaces.
This family of techniques has lead to several spectacular successes in elucidating the structure of mapping class groups. Particularly interesting for the purposes of this paper, Mosher has constructed an automatic (though not bi-automatic) structure on all mapping class groups~\cite{MosherAutomatic}, and he has given a solution to the conjugacy problem in mapping class groups~\cite{MosherConjugacy} (in the pseudo-Anosov case). Masur and Minsky~\cite{MasurMinsky2} together with Tao~\cite{Tao} have proved a linear bound on the conjugator length in mapping class groups, which implies an exponential time solution to the conjugacy problem (see also \cite{ATW}). More recently, Bell and Webb~\cite{BellWebbPolyn} have constructed (and implemented~\cite{curver}) an algorithm for solving in polynomial time the closely related problem of determining the Nielsen--Thurston type and (in the reducible case) the canonical reduction system of any given mapping class. 
Also, Margalit, Strenner and Yurtta\c{s} have announced a quadratic time algorithm, using very different techniques, for determining the Nielsen--Thurston type, the canonical reduction system, and (in the pseudo-Anosov case) the stable and unstable foliations and their stretch factors. 
Most spectacularly, Bell and Webb have announced a polynomial time solution to the conjugacy problem in mapping class groups (where the constants of the polynomial depend exponentially on the complexity of the surface) -- again, details have not yet appeared.
\item 
Garside-theoretic approaches, which typically talk not only about the braid groups but more generally about Garside groups~\cite{DehornoyGroupesDeGarside}, and in particular about Artin groups of spherical type~\cite{CharneyBiautomatic}. This approach has yielded a bi-automatic structure (which we still don't know to exist on other mapping class groups), and it gives rise to simple algorithms for the conjugacy problem which work very fast in practice. Also, in any braid group~$B_N$, there is a cubic-time algorithm for determining the Nielsen-Thurston type of any given element~\cite{CalvezPolynTime}. 
\end{itemize}

Making these two approaches work together, creating a synergy between them, appears to be a difficult task. Here is a rather embarrassing illustration of this difficulty: let us consider the set of all Garside normal form words in~$B_N$, and let us also look at the curve complex of the $N$-times punctured disk, equipped with a base point. The traces of the base point under the action of the Garside words form a family of paths in the curve complex. It is currently an open problem whether these paths are a uniform family of reparametrised quasi-geodesics! 

In the present paper we propose another interaction, namely between the Garside-theoretic solution to the conjugacy problem and subsurface projections. 
Specifically, we study examples of braids where the Garside-theoretic conjugacy algorithm works relatively slowly, and we try to explain this lack of efficiency in terms of multiple subsurfaces having very large projections. Our analysis suggests that even these ``bad'' cases, which we conjecture to be the worst possible ones, are still quite satisfying, in that they would guarantee a polynomial bound on the computational complexity of the conjugacy problem.

In order to explain the results of this paper in more detail, we recall very briefly some ideas of Garside theory -- for more details, see Section~\ref{S:GarsideAndOuroboroi}.
All the classical Garside-theoretic solutions to the conjugacy problem in braid groups are based on the same principle: to every braid~$x$ one associates a certain finite subset of the conjugacy class of~$x$,  and this subset is characteristic in the sense that if $x$ and $x'$ are conjugate, then the subsets in question coincide. Now in order to decide whether two given braids are conjugate, it suffices to determine algorithmically the full characteristic subset of one, and at least one element of the characteristic subset of the other, and check whether the latter belongs to the former.

Over time, various characteristic subsets have been proposed (the summit set, the super summit set, the ultra summit set...) \cite[Section 3.2]{Gebhardt-GM}, but we will use the \emph{sliding circuit set}~$SC(x)$ defined in~\cite{Gebhardt-GM}. 
We will not give the definition for general braids~$x$, but only in the case where~$x$ is \emph{rigid}. Roughly speaking, this means that the Garside normal form of~$x^2$ is equal to two copies of the Garside normal form word of~$x$, concatenated -- in other words, the last letter of the normal form, followed by the first letter, is again in normal form. It is a theorem of Gebhardt and Gonz\'alez-Meneses~\cite{Gebhardt-GM} that if $x$ is conjugate to a rigid braid, then $SC(x)$ consists exactly of all rigid conjugates of~$x$.

The main difficulty proving a polynomial bound on the computational complexity of the conjugacy problem (using the classical Garside-theoretic strategy, and in the pseudo-Anosov case) is establishing a polynomial bound on the size of the sliding circuit set of rigid braids, as a function of braid length, for a fixed number of strands. (In fact, for reducible, non-rigid braids it is known that the size of $SC$ \emph{can} grow exponentially with the braid length~\cite{GonzMenRed}.) 

In this paper we present some examples of braids with remarkably large sliding circuit sets:

\begin{theorem}\label{T:Main}
There exists a family of positive braids $\gamma(N,L)$, with $N$ strands and of Garside-length~$L$ (with $N$~even, and $N\geqslant 4$, and $L$~odd, $L\geqslant N-1$) such that  $$|SC(\gamma(N,L))| = 2\cdot (L-1)\cdot L^{N-3}$$
\end{theorem}

For $N=5$ we have examples of braids whose sliding circuit sets are even slightly larger, namely of size $2L^3=2L^{N-2}$. Also, we have some sporadic exemples of very short braids (with $L\leqslant 3$) whose sliding circuit set has strictly more than $2L^{N-2}$ elements. However, we propose

\begin{conjecture}[Polynomial bounds on the sliding circuit set]\label{C:BoundOnSC}
There exists a constant~$C$ such that for any rigid braid~$x$ with $N$ strands and Garside-length~$L$, 
$$
|SC(x)| \leqslant C\cdot L^{N-2}
$$
We even conjecture that the value $C=2$ is valid for sufficiently large~$L$. (Note that we are not supposing $x$ to be pseudo-Anosov.)
\end{conjecture}

More interesting than the precise size of the sliding circuit set in Theorem~\ref{T:Main} are the \emph{geometric} properties of the example braids~$\gamma(N,L)$. 
In all the examples which we found (in part through computer searches), of long braids with very unusually large sliding circuit sets, 
the size of~$SC$ is readily explained by the presence of multiple \emph{ouroboroi}. We will  give a rigorous definition of an ouroboros later, but roughly speaking, an ouroboros is a subsurface of the $n$-times punctured disk which is almost invariant under the action of the braid. Thus the ``bad'' braids in question are very close to being reducible -- for instance, they act on the curve complex of~$D_N$ with very small translation distance. Moreover, in our examples the largest sliding circuit sets occur when different ouroboroi move relative to each other -- we say they ``slither''. This situation is strongly reminiscent of disjoint subsurfaces with large projections.

We conjecture that the presence of ouroboroi is essentially the \emph{only} reason sliding circuit sets can become big. The following very vague conjecture will be made more precise later (Conjecture~\ref{C:BigSCisGeometric}).

\begin{conjecture}[Commutativity conjecture]\label{C:CommutativityConjecture}
At least for sufficiently long braids, large sliding circuit sets come from some kind of internal commutativity of the braid, and this internal commutativity is a geometric fact: if a braid has a big sliding circuit set, then it has several slithering ouroboroi.
\end{conjecture}

If some reasonable interpretation of the Commutativity Conjecture was proven, then we would probably obtain a polynomial bound on the size of the sliding circuit set, and hence on the complexity of the conjugacy problem in the braid group~$B_N$, for any fixed~$N$, at least in the pseudo-Anosov case.

The rest of the paper is organised as follows. In Section~\ref{S:GarsideAndOuroboroi} we briefly review some of the background in Garside theory; we also define ouroboroi, our main geometric tool. Section~\ref{S:Examples} contains our main results, namely the examples of braids with unusually large sliding circuit sets. The main tool for constructing such examples are ouroboroi. Finally, in Section~\ref{S:Conjectures}, which is more speculative, we present some ideas on the structure of sliding circuit sets. This will put into context our Commutativity Conjecture. 
Also, the conjectured structure could be helpful in attempts to find a \emph{uniform} polynomial time solution to the conjugacy problem.


\section{Garside theory and ouroboroi}\label{S:GarsideAndOuroboroi}

In this section we recall some important known results about the Garside-theoretic approach to the conjugacy problem in braid groups. Note that in this paper we will be interested in the case where the braids are assumed to be pseudo-Anosov, so this also will be the focus in this section. We will also give the definition of an \emph{ouroboros}, the new geometric tool which we will be using throughout the paper.

We recall that that every element of the braid group $B_N$ has a unique normal form $x=\Delta^k.x_1.x_2.\ldots. x_\ell$, where
\begin{enumerate}
\item $\Delta$ is the half-twist braid (whose square generates the center of~$B_N$).
\item Every $x_i$ is a \emph{simple} braid (also known as positive permutation braid or Garside braid), i.e. a positive braid such that any two stands cross at most once.
\item Every pair $x_i.x_{i+1}$ is left-weighted, meaning that $x_i$ contains as many crossings as possible, and $x_{i+1}$ as few crossings as possible, among all writings of the braid $x_i x_{i+1}$ as a product of two simple braids.
\end{enumerate}

After multiplying $x$ by an element of the center $\langle \Delta^2\rangle$, we can assume that $k=0$ or $k=1$, and in particular that $x$ is positive. For the rest of the paper we will only talk about positive braids.
The \emph{supremum} of such a braid is $\sup(x)=k+\ell$. We will often denote it~$L$, because in our context it is just the length of the braid word; and by the ``length'' of a braid we will always mean this Garside-length.

Next we recall the definition of a \emph{rigid} braid.
\begin{itemize}
\item If $k$ is even, we say $x$ is rigid if the writing $x_\ell.x_1$ is left-weighted, as well (i.e., if the \emph{cyclic} word $x_1.x_2.\ldots.x_\ell.$ is left-weighted everywhere) .
\item If $k$ is odd, we say $x$ is rigid if the writing $x_\ell.\Delta^{-1}x_1\Delta$ is left-weighted. (Note that $\Delta^{-1}x_1\Delta$ is again a simple braid)
\end{itemize}

Among the set of all positive braids on $N$ strands, there is a large class of braids~$y$, including all pseudo-Anosov braids, with the following property: they have a power~$y^k$ (with $k\leqslant \left(\frac{N(N-1)}{2}\right)^3$) which is conjugate to a rigid braid \cite{BirGebGM1}.
Moreover, there is an algorithm which, for any given~$y$, finds the appropriate power~$k$ and a rigid conjugate of $y^k$, if it exists. This algorithm works in polynomial time in the length of~$y$ (for fixed~$N$)~\cite{CalvezPolynTime}.

For a positive braid~$x$ which has a rigid conjugate (e.g.\ for $x=y^k$), the \emph{sliding circuit set} $SC(x)$ is the set of all rigid conjugates of~$y$. (This is actually a theorem of Gebhardt and Gonz\'alez-Meneses~\cite{Gebhardt-GM}, but for the purposes of the present paper, we can use this as the definition of the sliding circuit set.) The sliding circuit set is always finite, and in fact it is a subset of the well-known \emph{super summit set} of \cite{ElrifaiMorton}.

If we want to solve the conjugacy problem, then we need a computable, complete invariant of conjugacy classes. If $x$ is a positive braid with a rigid conjugate, then $SC(x)$ is such an invariant. 
Unfortunately, calculating the full sliding circuit set (not just one of its elements) may a priori be difficult.
For solving the conjugacy problem in braid groups in polynomial time (at least in the pseudo-Anosov case), it would be sufficient to place a polynomial bound on the number of elements in the sliding circuit set, as a function of the length of the input braid:

\begin{question}\label{Q:SCpolynBound?} 
Is it true that for every integer $N$ (with $N\geqslant 5$) there exists a polynomial~$P_N$ such that every positive rigid braid $x\in B_N$ with $\sup(x)=L$ satisfies
$$ |SC(x)| \leqslant P_N(L) \ ?$$
\end{question}

In the next section, we will present some examples of braids where the sliding circuit set is relatively large -- still of polynomially bounded size, but remarkably large nevertheless. We will show that in each case, the braid has a very particular structural feature, which we call an \emph{ouroboros}. In order to explain this choice of words we recall that an ``ouroboros'' usually means a dragon or a snake biting its own tail -- these creatures appeared in mythologies of several cultures.

Before defining an ouroboros, we recall a result of Bernardete, Nitecki and Guti\'errez~\cite{BNG}. We denote $D^2$ the unit disk in~$\mathbb C$, and $D_N$ the same disk, but with $N$ punctures lined up on the real line. We say a simple closed curve in~$D_N$ (or its isotopy class) is \emph{round} if it is isotopic to a circle in~$D_N$ that contains at least two, but not all the punctures. If the action of a braid~$x$ sends a round curve~$c$ to a round curve~$c'$, and if the Garside normal form of~$x$ is $\Delta^k x_1.\ldots.x_\ell$, then, according to~\cite{BNG}, every prefix $\Delta^k x_1.\ldots.x_{\tilde \ell}$ (with $\tilde \ell < \ell$) also sends $c$ to some round curve.

In the following definition we take $x$ to be a braid with normal form $\Delta^k x_1\ldots x_\ell$, where $k\in\{0,1\}$. We denote $L$ the total number of factors: $L=k+\ell=\sup(x)$.

We take $x$ to be realised as a braid in the solid cylinder $D^2\times [0,L]$, with the factor $x_i$ living in $D^2\times [i,i+1]$. The closure $\hat x$ is realised in the solid torus $(D^2\times [0,L])/\sim$, where $\sim$ is the equivalence relation $(x,0)\sim (x,L)$ for all $x\in D^2$. 

We are going to define two types of ouroboroi, \emph{round} and \emph{eccentric} ones. Before we can do so, we have to define their \emph{base curves}.

\begin{definition}
The \emph{base curve of a round ouroboros} with $m$ strands and \emph{head-tail} in the $i$th factor is a round curve~$c$ containing $m$~punctures in its interior, and which is sent to a round curve by the action of $x_{i+1} x_{i+2} \ldots x_\ell \Delta^k x_1 x_2 \ldots x_{i-1}$.
Moreover, we require that the intersection of the disks bounded by $c$ and by $c.x_{i+1} \ldots x_\ell \Delta^k x_1 \ldots x_i$ consists of a single disk which contains at most~$m-2$ punctures. 

The \emph{base curve of an eccentric ouroboros} with $m$ strands and \emph{head-tail} in  factors $x_i$ and $x_{i+1}$ is a round curve~$c$ containing $m$ punctures in its interior, such that the curve $c.x_{i+2}\ldots x_\ell \Delta^k x_1\ldots x_{i-1}$ is again round. Moreover, we require that the disks bounded by the curves $c.x_{i+1}^{-1}$ and $c.x_{i+2}\ldots x_\ell \Delta^k x_1\ldots x_i$ contain the same punctures.
\end{definition}

By~\cite{BNG}, the image of~$c$ after each intermediate factor of the normal form is again a round curve. Thus the curve~$c$ induces a round tube going almost completely around the braid. Only in the one or two head-tail factors does the shape of the tube get slightly more complicated.
This tube is what we will call an ouroboros.
Notice that the tube does not close up into a torus -- if it did, the braid would be reducible. So having an ouroboros is very close to being reducible. Here is the formal definition:

\begin{definition}
A \emph{round ouroboros} of~$x$, with $m$ strands and \emph{head-tail} in the $i$th factor is a cylinder which is properly embedded in $(D^2\times ([i+1,L]\cup [0,i]))/\sim$, disjointly from the braid, and whose intersection with the disk $D^2\times \{i+1\}$ is a base curve~$c$ of a round ouroboros as defined above.
(Its intersection with $D^2\times \{i\}$ is thus the round curve $c.x_{i+1} x_{i+2} \ldots x_\ell \Delta^k x_1 x_2 \ldots x_{i-1}$.)

An \emph{eccentric ouroboros} of~$x$ is defined analogously, as a cylinder properly embedded in $(D^2\times ([i+2,L]\cup [0,i]))/\sim$.
\end{definition}

\begin{figure}[htb]
\centering
\def\svgwidth{12cm}
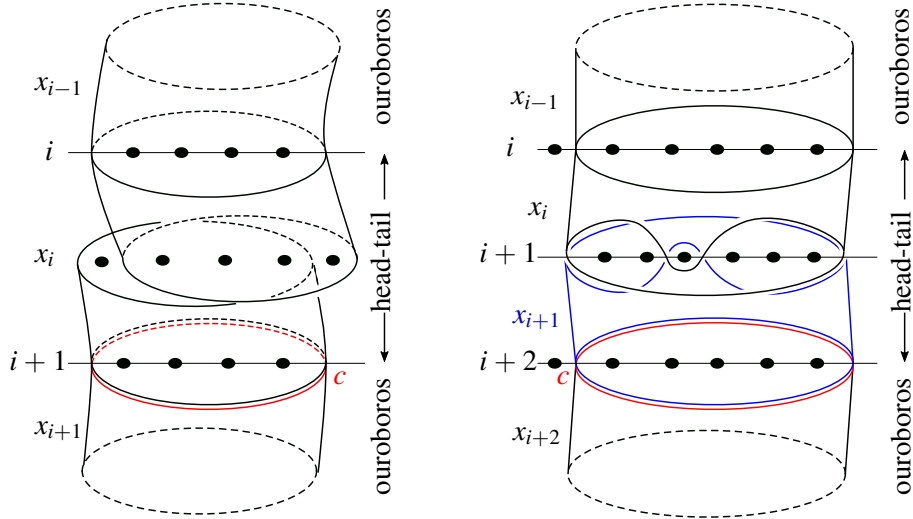
\caption{The head-tail factors of a round (left) and an eccentric (right) ouroboros. The base curves are indicated in red and labelled~$c$.  Also, as a visual aid, on the right hand side the factor~$x_{i+1}$ and the corresponding part of the ouroboros are drawn in blue.}
\end{figure}

\begin{remark}
Let us say a curve is \emph{almost round} if it is not round, but it is the image of a round curve under the action of a simple braid. If a braid~$x$ has a round ouroboros, then the image of the round curve~$c$ under the action of $x_{i+1} \ldots x_\ell \Delta^k x_1 \ldots x_i$ is an almost round curve. If $x$ possesses an eccentric ouroboros, then the image of the almost round curve $c.x_{i+1}^{-1}$ under the action of $x_{i+1}\ldots x_\ell \Delta^k x_1\ldots x_i$ is the almost round curve $c.x_{i+2}\ldots x_\ell \Delta^k x_1\ldots x_{i-1}$.
In both cases, a conjugate of $x$~sends some round or  almost round curve to an almost round curve. Now, any two almost round curves in~$D_n$ are at distance at most~3 in the curve complex (this is an exercise -- see \cite[Section 4.1]{FMPrimer} for an introduction to the curve complex). Thus, for any braid~$x$ with an ouroboros, the action of~$x$ on the curve complex has translation distance at most~$3$. 
\end{remark}


\section{Large sliding circuit sets and ouroboroi: examples}\label{S:Examples}

In this section we present some families of examples of braids with unusually large sliding circuit sets. We show how the presence of multiple ouroboroi allows the braid to have such a large sliding circuit set. On the other hand, we also see that this way of creating large sliding circuit sets can only yield sliding circuit sets whose size grows polynomially with the length of the braid, not exponentially. 

The braids presented in this section are the worst we found during large computer-searches for braids with big sliding circuit sets. This may be interpreted as evidence that Conjectures~\ref{C:BoundOnSC} and \ref{C:CommutativityConjecture} are true, and thus that the conjugacy problem from pseudo-Anosov braids can be solved in polynomial time.

The families of braids studied in Subsections~\ref{SS:SimpleEx} and~\ref{SS:NestedOuroboroi} depend on two parameters $N$ (the number of strands) and $L$ (the Garside length), and the family in Subsection~\ref{SS:EccentricExample} on the single parameter~$L$. Our main tool for finding and studying in detail these braids was the computer program~\cite{JuanProgram}. In each of these families of examples, the size of the sliding circuit set can, in principle, be determined by hand, with a formal proof. However, the number of case checks involved is prohibitive. 
We take an experimental approach: we calculate the size of the sliding circuit sets for a substantial number of pairs $(N,L)$, trusting the program~\cite{JuanProgram} to give the correct values. Also, from the calculated values we extrapolate, and give general formulae (for all possible values of $N$ and~$L$) for the size of the sliding circuit set. While this is, in principle, an unreliable methodology, we believe that most readers will be convinced by our extrapolations in these three particular cases.


\subsection{A simple example introducing ouroboroi}\label{SS:SimpleEx}

For any odd number~$N$ ($N\geqslant 5$) and any integer $L$ ($L\geqslant 2$) consider the braid with $N$~strands
$$\beta(N,L) = (\sigma_1 \sigma_3 \sigma_5 \ldots \sigma_{N-2}.)^{L+1} \sigma_1 \sigma_2 \sigma_3 \ldots \sigma_{N-1}$$ 
This braid of length~$L+2$ is not rigid, but it is conjugate to a rigid braid of length~$L$. Hence $SC(\beta(N,L))$ consists exactly of the rigid conjugates of~$\beta(N,L)$, which are all of length~$L$. Since these rigid braids are not powers of any other elements, we have that
$$S(N,L):=\frac{|SC(\beta(N,L))|}{L}$$
is equal to the number of orbits under cycling (cyclic permutation of the factors) of $SC(\beta(N,L))$. 
Using the program~\cite{JuanProgram} we calculated the value $S(N,L)$ for many pairs $(N,L)$: 
\begin{small}
\begin{center}
\begin{tabular}{c||c|c|c|c|c|c|c|c|c|c|c}
$N \backslash L$ & 2 & 3 & 4 & 5 & 6 & 7 & 8 & 9 & 10 & 11 & 12 \\ \hline\hline
5 & 1 & 2 & 3 & 4 & 5 & 6 & 7 & 8 & 9 & 10 & 11 \\
7 & 2 & 6 & 12 & 20 & 30 & 42 & 56 & 72 & 90 & 110 & 132 \\ 
9 & 4 & 16 & 40 & 80 & 140 & 224 & 336 & 480 & 660 & 880 & 1144 \\
11 & 8 & 40 & 120 & 280 & 560 & 1008 & 1680 & 2640 & 3960 & 5720 & 8008\\
13 & 16 & 96 & 336 & 896 & 2016 & 4032 & 7392 & \small{12672} &  &  & \\
15 & 32 & 224 & 896 & 2688 & 6720 & 14784 &  &  &  &  & \\
\end{tabular}
\end{center}
\end{small}

The numbers occurring in this table are exactly those appearing in the table in A080928 of the Online Encyclopedia of Integer Sequences~\cite{OEIS}. This leads us to conjecture that with $n=\frac{N-3}{2}$, with $\lambda=L-2$, and with $T$ the function defined in A080928 of the OEIS, we have 
$$
S(N,L) = T(n+\lambda,\lambda) = {n+\lambda \choose n}\cdot 2^{n-1}={\frac{N-3}{2}+L-2\choose \frac{N-3}{2}}\cdot 2^{\frac{N-5}{2}}
$$

(Reference \cite{OEIS} also gives an amusing recurrence relation: for $N\geqslant 7$ and $L\geqslant 3$, 
$
S(N,L)=S(N,L-1)+2\cdot S(N-2,L)
$
where $S(5,L)=L-1$ and $S(N,2)=2^{\frac{N-5}{2}}$.) However, the important conclusion for us is that
for fixed~$N$, the function $S(N,L)$ in the variable $L$ is a polynomial of degree $\frac{N-3}{2}$. Since $|SC(\beta(N,L))|=L\cdot S(N,L)$, this implies:

\begin{observation}\label{O:SchlWie}
Let $N$ be an odd integer with $N\geqslant 5$. Then the size of the sliding circuit set $|SC(\beta(N,L))|$, seen as a function of the variable $L$ (with $L\in\mathbb N$, $N\geqslant 2$), is a polynomial of degree~$\frac{N-1}{2}$. 
\end{observation}

For instance, for $N=5$ we obtain that $|SC(\beta(5,L))|=L(L-1)$, similarly $|SC(\beta(7,L))|=L^2(L-1)$, and $|SC(\beta(9,L))|=\frac23 L^2(L+1)(L-1)$.

As mentioned in the introduction to this section, we are not going to give a formal proof of the equality $S(N,L)=T(n+\lambda,\lambda)$ and of Observation~\ref{O:SchlWie}. 
However, in order to understand where the polynomial growth comes from, we look at the example $N=7$ and actually prove:

\begin{proposition}\label{P:SWlowerbound} 
If $L\geqslant 3$, the sliding circuit set $SC(7,L)$ has at least $\frac{L(L-1)(L-2)}{2}$ elements.
\end{proposition}

\begin{proof} 
If $L\geqslant 3$, the sliding circuit set $SC(\beta(7,L))$ contains the elements of length~$L$
$$
\sigma_2 \sigma_1 \sigma_4 \sigma_6. (\sigma_1 \sigma_4 \sigma_6 .)^a  \sigma_1 \sigma_4 \sigma_3 \sigma_5 \sigma_6 \sigma_5. ( \sigma_1 \sigma_3 \sigma_5.)^b  \sigma_5 \sigma_4 \sigma_3 \sigma_2 \sigma_1 \sigma_6 (. \sigma_2 \sigma_4 \sigma_6)^c
$$ 
and their cyclic conjugates. Here $(a,b,c)$ are all triples of integers between $0$ and $L-3$ with $a+b+c=L-3$. 
(The left hand side of Figure~\ref{F:SchlWieEx} shows the example with $a=2, b=3, c=3$.) There are ${L-1 \choose 2}=\frac{(L-1)(L-2)}{2}$ such triples. If we also count the cyclic conjugates of these braids, which are all different, we find a total of $L\cdot \frac{(L-1)(L-2)}{2}$ elements
\end{proof}

\begin{figure}[htb] 
\begin{center}
\includegraphics[width=100mm]{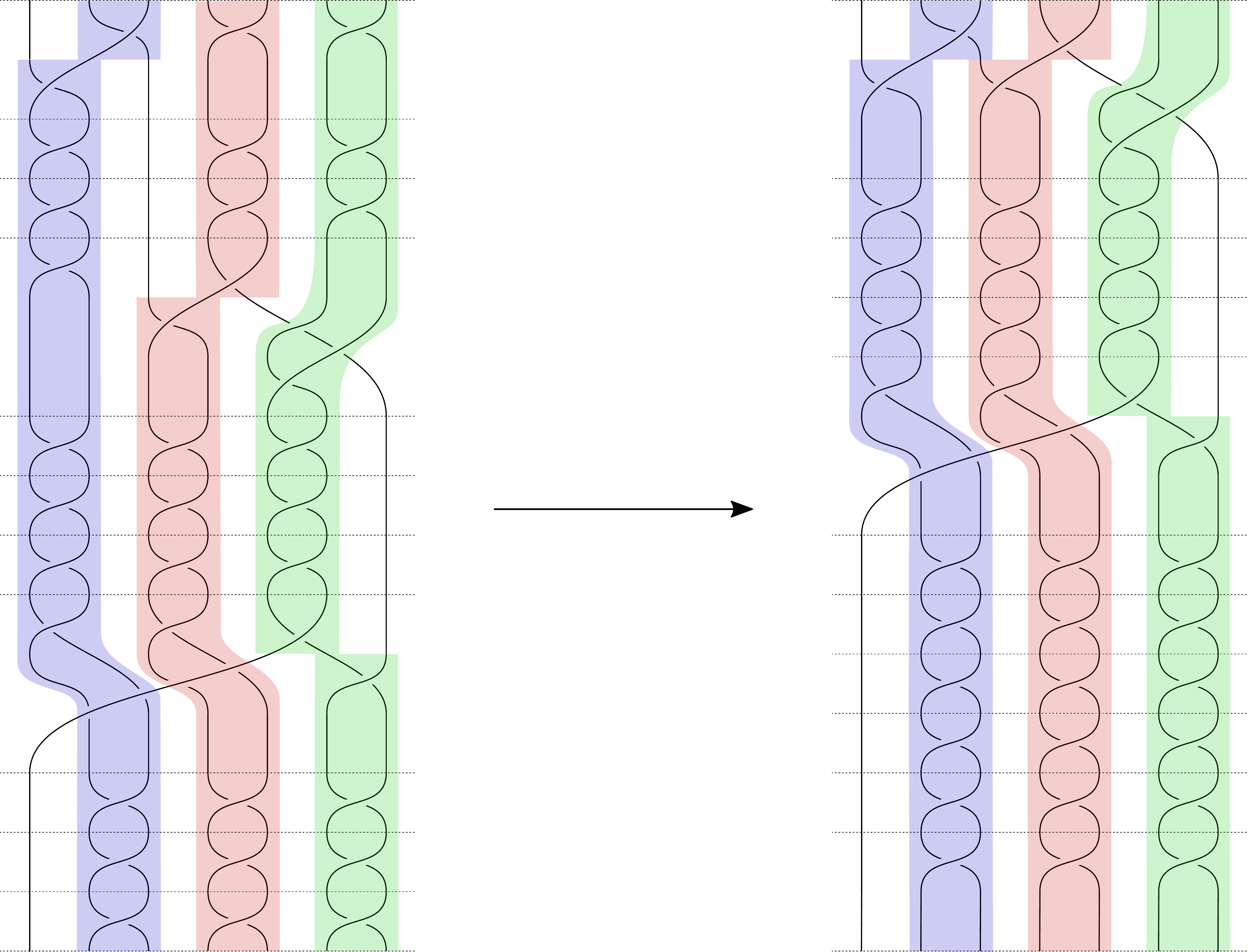}
\end{center}
\caption{Slithering ouroboroi in $SC(\beta(7,11))$. The arrow represents conjugation by $(\sigma_4\sigma_6)^3$}\label{F:SchlWieEx}
\end{figure}

Let us study Figure~\ref{F:SchlWieEx}, i.e.\ the example of $\beta(7,11)$, in more detail. This example is very instructive, because it is very easy to see three ouroboroi, shown in blue, red, and green in the figure. Let us denote $\beta$ the braid on the left hand side of the figure, where $a=2, b=3$, and $c=3$. Conjugating $\beta$ by positive or negative powers of $\sigma_2$, $\sigma_4$ and $\sigma_6$ has the effect of varying the coefficients $a$, $b$, and $c$, which, visually, corresponds to ``slithering the ouroboroi''.

For instance, conjugating $\beta$ by $(\sigma_4 \sigma_6)^k$ for $k=1$ or~$2$ slithers the red and green ouroboros upwards simultaneously. We can even conjugate by $(\sigma_4 \sigma_6)^3$, which creates the braid on the right hand side of the figure, where the blue and red head-tail occur in the same factor. The red ouroboros cannot be slithered upwards any further relative to the blue one, because it is blocked above. The green one can slither up three more factors, but a detailed calculation shows that the three head-tails can never occur all in the same factor.  Similarly, starting with~$\beta$ and conjugating by $\sigma_4^{-1}$ three times slithers the red ouroboros down relative to the blue and green ones -- at which point the red ouroboros gets stuck against the green one. Intuitively, the two degrees of freedom for the relative position of the ouroboroi yields the quadratic growth of $S(7,L)$.

\begin{remark}\begin{enumerate}
\item The three ouroboroi have exactly the same length. If this requirement is violated, the growth of the sliding circuit set (as a function of the length) goes down. For instance, we can modify the definition of $\beta(7,L)$, and declare $\widetilde\beta(7,L)=(\sigma_1 \sigma_3 \sigma_5.)^L \ \sigma_1 \sigma_3. \ \sigma_1\sigma_2\sigma_3\sigma_4\sigma_5\sigma_6$ -- i.e., remove one crossing from the green ouroboros. Then $|SC(\widetilde\beta(7,L))|$ grows only quadratically with~$L$. 
\item For $L=0$, the braids $\beta(N,L)=(\sigma_1 \sigma_3 \sigma_5 \ldots \sigma_{N-2}.)^{L+1} \sigma_1 \sigma_2 \sigma_3 \ldots \sigma_{N-1}$ are conjugate to rigid but \emph{reducible} braids, with remarkably large sliding circuit sets: for $N=5, 7, 9, \ldots, 17$ we find sliding circuit sets of size  4, 2, 30, 294, 2520, 20460, 162162.
For $L=1$ the braids $\beta(N,L)$ are pseudo-Anosov, but not conjugate to rigid braids, and their sliding circuit sets have $(N+3)\cdot 2^{\frac{N-7}2}$ elements.
\end{enumerate}
\end{remark}


\subsection{An example with nested ouroboroi}\label{SS:NestedOuroboroi}

In this section we will exhibit an example of a family of rigid pseudo-Anosov braids whose sliding circuit sets grow remarkably fast as a function of both the number of strands and of the length of the braid. The basic trick will be to pack the ouroboroi extremely tightly by nesting them inside each other.

Our family of braids $\gamma(N,L)$ can have any even number $N\geqslant 4$ of strands and any odd number $L=N-1+2 K$ (with $K\geqslant 0$) of Garside factors. (More precisely, the first Garside factor is $\Delta$, so our braids have $\inf(\gamma(N,L))=1$ and $\sup(\gamma(N,K))=N-1+2 K$.) The main result of this section is Theorem~\ref{T:Main} which we recall now: 

\begin{unnumberedtheorem} The braids $\gamma(N,L)$ are rigid pseudo-Anosov braids, and their sliding circuit sets have size
$$
|SC(\gamma(N,L))|=2(L-1)\cdot L^{N-3}
$$
More precisely, there are $L^{N-3}$ orbits under the cycling operation, and each of these orbits has $2(L-1)$ elements. There is one exception: for $N=4, L=3$ a particular symmetry occurs, resulting in $|SC(\gamma(4,3)|=6$, rather than 12.
\end{unnumberedtheorem}

As mentioned in the introduction to this section, the formal proof of this result is tedious, and is omitted.
We will give the results of our calculations using the program~\cite{JuanProgram}, and we assert that the formulae observed in the calculated examples continue to hold for general values of $N$ and~$L$.

We now explain the construction of the braid $\gamma(N,L)$. It has $L=N-1+2K$  factors, which come in three blocks:
\begin{enumerate}
\item first one factor $\Delta$
\item then $N-2$ factors which we call head-tail factors, because they correspond to the head-tails of $N-2$ ouroboroi, and finally
\item two Garside factors, which we call the body-factors because they correspond to the bodies of all the ouroboroi, which are repeated $K$ times.
\end{enumerate}

The first block requires no explanation. In order to present the other blocks, we introduce some notation. 

For any fixed integer~$N$, we say a \emph{down-up sequence} is a writing of the numbers $1,2,\ldots,N$ permuted in such a way that it consists of a (possibly empty) decreasing sequence of numbers, followed by a (possibly empty) increasing sequence of numbers. For instance, for $N=10$, the sequences $(8\ 7\ 4\ 3\ 1\ 2\ 5\ 6\ 9\ 10)$ and $(9\ 8\ 5\ 3\ 2\ 1\ 4\ 6\ 7\ 10)$ are down-up sequences, whereas $(10\ 8\ 5\ 7\ 2\ 1\ 3\ 4\ 6\ 9)$ is not. We call the numbers occurring in such a sequence \emph{labels}, because we will use them to label strands of a braid.

For any subset $A\subseteq \{1,2,\ldots N\}$ there is an involution $\phi_A$ on the set of  down-up sequences, defined as follows: if $S$ is a down-up sequence, then in the down-up sequence $\phi_A(S)$, every label belonging to~$A$ switches its position relative to all lower labels, whereas all labels in the complement of~$A$ retain their position relative to all lower labels. We say ``All labels in~$A$ switch sides''. For instance, 
if $N=8$ and $A=\{2, 4, 6, 8\}$, then
$$
(8\ 7\ 4\ 3\ 2\ 1\ 5\ 6) \ \stackrel{\phi_A}{\longmapsto} \ (7\ 6\ 3\ 1\ 2\ 4\ 5\ 8)  \ \stackrel{\phi_A}{\longmapsto} \ (8\ 7\ 4\ 3\ 2\ 1\ 5\ 6)
$$

Now in order to define our braid $\gamma(N,L)$ with $N$ strands and $L=N-1+2K$ factors, we first define a sequence of $L+1=N+2K$ down-up sequences.  For the example $N=8$, see Figure~\ref{F:BadSymbolic} (ignoring the fat line segments for the moment). 

The first sequence consists of the integers between 1 and $N$ congruent to 0 or 3 modulo~4 in descending order, followed by the remaining integers in ascending order. The second sequence (at the junction between the first block and the second block) is just the reverse of the first sequence (all labels have switched sides).

Concerning the second block: here each sequence is obtained from the preceding one by making all even labels switch sides -- except that in each step, one label behaves in an unexpected way. Specifically, starting from the second sequence, we obtain the third one by making all even labels \emph{except $N$} switch sides. We go from the third to the fourth sequence by making all even labels \emph{and label $N-1$} switch sides, and so on through labels $N, N-1, \ldots, 4, 3$. Thus the $N$th sequence (at the junction between the second and third block) is identical to the first one, except that labels 1 and 2 have exchanged their positions. (We recall that $N$ is even.)

The third block is simple again: we go from each down-up sequence to the next by making all even labels switch sides -- thus we simply go back and forth between two sequences $K$~times. This completes our description of a sequence of down-up sequences. 

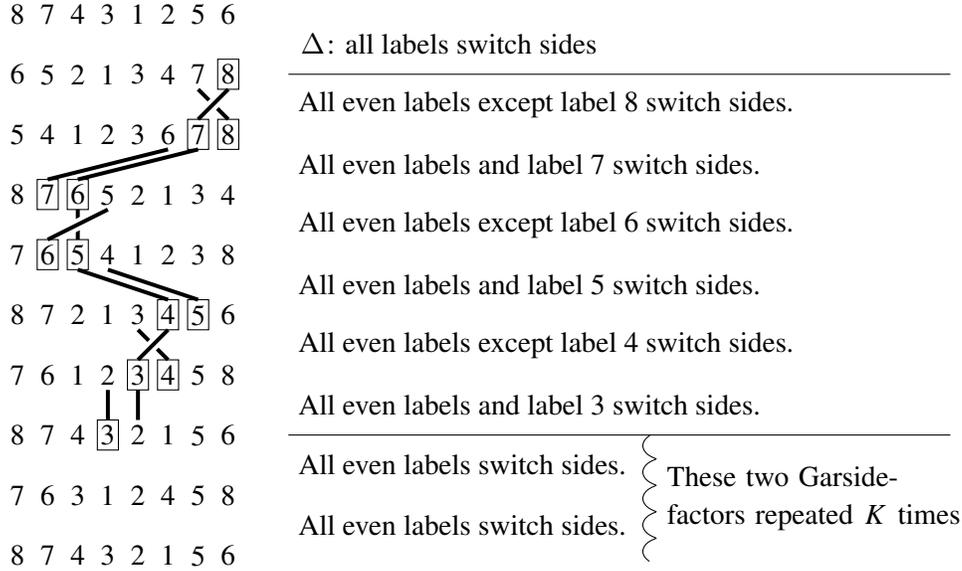
\begin{figure}[htb] 
\pgfdeclarelayer{deepest}
\pgfdeclarelayer{middle}
\pgfsetlayers{deepest,middle,main}
\begin{tikzpicture}[scale=0.4]
\draw (1,18) node{8};
\draw (2,18) node{7};
\draw (3,18) node{4};
\draw (4,18) node{3};
\draw (5,18) node{1};
\draw (6,18) node{2};
\draw (7,18) node{5};
\draw (8,18) node{6};

\draw (10,17) node[right]{$\Delta$: all labels switch sides};

\draw (1,16) node{6};
\draw (2,16) node{5};
\draw (3,16) node{2};
\draw (4,16) node{1};
\draw (5,16) node{3};
\draw (6,16) node{4};
\draw (7,16) node{7};
\draw (8,16) node{8}; \draw (8,16)+(-\bx,-\by) rectangle +(\bx,\by);
\draw (10,16) -- (32,16);

\draw[line width = \Bdiag] (7,14+\by) -- (8,16-\by);
\begin{pgfonlayer}{middle} \draw[white,line width = 3*\Bdiag] (7,14+\by) -- (8,16-\by);  \end{pgfonlayer}
\begin{pgfonlayer}{deepest} \draw[line width = \Bdiag] (8,14+\by) -- (7,16-\by); \end{pgfonlayer}
\draw (10,15) node[right]{All even labels except label 8 switch sides.};

\draw (1,14) node{5};
\draw (2,14) node{4};
\draw (3,14) node{1};
\draw (4,14) node{2};
\draw (5,14) node{3};
\draw (6,14) node{6};
\draw (7,14) node{7}; \draw (7,14)+(-\bx,-\by) rectangle +(\bx,\by);
\draw (8,14) node{8}; \draw (8,14)+(-\bx,-\by) rectangle +(\bx,\by);

\draw[line width = \Bdiag] (2,12+\by) -- (6,14-\by);
\draw[line width = \Bdiag] (3,12+\by) -- (7,14-\by);
\draw (10,13) node[right]{All even labels and label 7 switch sides.};

\draw (1,12) node{8};
\draw (2,12) node{7}; \draw (2,12)+(-\bx,-\by) rectangle +(\bx,\by);
\draw (3,12) node{6}; \draw (3,12)+(-\bx,-\by) rectangle +(\bx,\by);
\draw (4,12) node{5};
\draw (5,12) node{2};
\draw (6,12) node{1};
\draw (7,12) node{3};
\draw (8,12) node{4};

\draw[line width = \Bdiag] (2,10+\by) -- (4,12-\by);
\begin{pgfonlayer}{middle} \draw[white,line width = 3*\Bdiag] (2,10+\by) -- (4,12-\by);  \end{pgfonlayer}
\begin{pgfonlayer}{deepest} \draw[line width = \Bdiag] (3,10+\by) -- (3,12-\by); \end{pgfonlayer}
\draw (10,11) node[right]{All even labels except label 6 switch sides.};

\draw (1,10) node{7};
\draw (2,10) node{6}; \draw (2,10)+(-\bx,-\by) rectangle +(\bx,\by);
\draw (3,10) node{5}; \draw (3,10)+(-\bx,-\by) rectangle +(\bx,\by);
\draw (4,10) node{4};
\draw (5,10) node{1};
\draw (6,10) node{2};
\draw (7,10) node{3};
\draw (8,10) node{8};

\draw[line width = \Bdiag] (6,8+\by) -- (3,10-\by);
\draw[line width = \Bdiag] (7,8+\by) -- (4,10-\by);
\draw (10,9) node[right]{All even labels and label 5 switch sides.};

\draw (1,8) node{8};
\draw (2,8) node{7};
\draw (3,8) node{2};
\draw (4,8) node{1};
\draw (5,8) node{3};
\draw (6,8) node{4}; \draw (6,8)+(-\bx,-\by) rectangle +(\bx,\by);
\draw (7,8) node{5}; \draw (7,8)+(-\bx,-\by) rectangle +(\bx,\by);
\draw (8,8) node{6};

\draw[line width = \Bdiag] (5,6+\by) -- (6,8-\by);
\begin{pgfonlayer}{middle} \draw[white,line width = 3*\Bdiag] (5,6+\by) -- (6,8-\by);  \end{pgfonlayer}
\begin{pgfonlayer}{deepest} \draw[line width = \Bdiag] (6,6+\by) -- (5,8-\by); \end{pgfonlayer}
\draw (10,7) node[right]{All even labels except label 4 switch sides.};

\draw (1,6) node{7};
\draw (2,6) node{6};
\draw (3,6) node{1};
\draw (4,6) node{2};
\draw (5,6) node{3}; \draw (5,6)+(-\bx,-\by) rectangle +(\bx,\by);
\draw (6,6) node{4};\draw (6,6)+(-\bx,-\by) rectangle +(\bx,\by);
\draw (7,6) node{5};
\draw (8,6) node{8};

\draw[line width = \Bdiag] (4,6-\by) -- (4,4+\by);
\draw[line width = \Bdiag] (5,6-\by) -- (5,4+\by);
\draw (10,5) node[right]{All even labels and label 3 switch sides.};

\draw (1,4) node{8};
\draw (2,4) node{7};
\draw (3,4) node{4};
\draw (4,4) node{3}; \draw (4,4)+(-\bx,-\by) rectangle +(\bx,\by);
\draw (5,4) node{2};
\draw (6,4) node{1};
\draw (7,4) node{5};
\draw (8,4) node{6};
\draw (10,4) -- (32,4);

\draw (10,3) node[right]{All even labels switch sides.};

\draw (1,2) node{7};
\draw (2,2) node{6};
\draw (3,2) node{3}; 
\draw (4,2) node{1};
\draw (5,2) node{2};
\draw (6,2) node{4};
\draw (7,2) node{5};
\draw (8,2) node{8};

\draw (10,1) node[right]{All even labels switch sides.};

\draw (1,0) node{8};
\draw (2,0) node{7};
\draw (3,0) node{4};
\draw (4,0) node{3};
\draw (5,0) node{2};
\draw (6,0) node{1};
\draw (7,0) node{5};
\draw (8,0) node{6};
\draw[decorate, decoration=coil] (22,-0.2) -- node[right=0.1cm,black, text width=3.9cm]{These two Garside-factors repeated $K$ times} (22,4);
\end{tikzpicture}
\caption{A symbolic picture of the braid $\gamma(8,7+2K)$. A box \fbox{\it i} means that the label $i$ is behaving unexpectely: an odd label~$i$ switching sides, or an even label~$i$ \emph{not} switching sides. Here is how to obtain a braid from this picture: the strands of the braid connect equal labels, except where indicated by bold lines.}\label{F:BadSymbolic}
\end{figure}

How do we obtain a braid from this sequence of down-up sequences? We recall that positive permutation braids (our generating set of the braid group) are in bijection with the permutations of $N$ symbols. The basic rule for constructing our braid is the obvious one: in each step from one down-up sequence to the next we use the positive braid which induces the given permutation. In other words, we connect equal labels by a string -- for an example with $N=6$, see Figure~\ref{F:BadBraid} with the values $k_1=K$ and $k_2=k_3=k_4=0$.

In the first and third block, this is indeed precisely how we construct the braid. However, in the second block there is one exception in every factor. Namely, in the factor in which label $i$ is behaving unexpectedly, we choose to have no strand connecting label $i$ to label $i$ and label $i-1$ to label $i-1$, but the two are exchanged: there are two strands connecting labels $i$ and $i-1$. This exceptional behaviour is symbolically indicated in Figure~\ref{F:BadSymbolic}, and it is shown with black strands in Figure~\ref{F:BadBraid} (which should again be taken with $k_1=K$ and $k_2=k_3=k_4=0$). This completes our description of the braid~$\gamma(N,L)$.

\begin{figure} 
\input{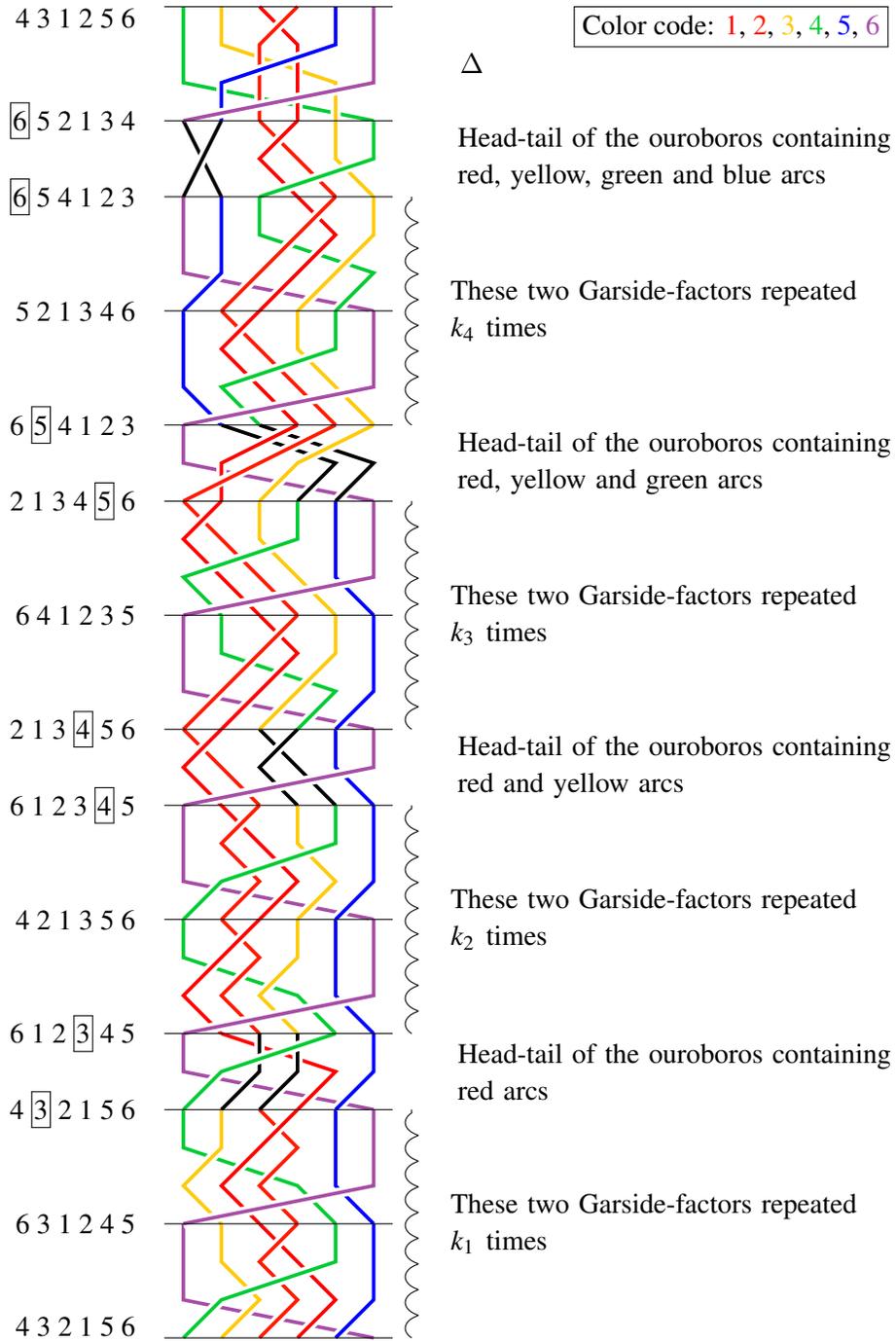}
\caption{The braid $\gamma(8,7+2K)$ is obtained with $k_1=K$ and $k_2=k_3=k_4=0$.}
\label{F:BadBraid}
\end{figure}

\begin{example} Let us write down the braids $\gamma(N,L)$ explicitly for $N=4$ and $N=6$ using Artin generators. Keeping in mind that $L=N-1+2K$, we obtain
\begin{itemize}
\item $\gamma(4,L)=\Delta_4$ . $\sigma_1 \sigma_3$ . $\sigma_1 \sigma_2 \sigma_3 \sigma_2 \sigma_1$ . $(\sigma_1 \sigma_2 \sigma_3 \sigma_2$ . $\sigma_2 \sigma_3 \sigma_2 \sigma_1 .)^K$
\item $\gamma(6,L)=\Delta_6$ . $\sigma_1 \sigma_3 \sigma_5 \sigma_4 \sigma_3$ . $\sigma_1 \sigma_2 \sigma_3 \sigma_2 \sigma_1 \sigma_4 \sigma_3 \sigma_2 \sigma_1 \sigma_5 \sigma_4 \sigma_3$ . $\sigma_1 \sigma_3 \sigma_5 \sigma_4 \sigma_3 \sigma_2 \sigma_1$~. $\sigma_1 \sigma_2 \sigma_1 \sigma_3 \sigma_2 \sigma_4 \sigma_3 \sigma_2 \sigma_1 \sigma_5$ . $(\sigma_1 \sigma_2 \sigma_3 \sigma_2 \sigma_5 \sigma_4 \sigma_3 \sigma_2 \sigma_1$ . $\sigma_1 \sigma_2 \sigma_3 \sigma_2 \sigma_4 \sigma_3 \sigma_2 \sigma_1 \sigma_5 .)^K$
\item For the sake of brevity, we are not discussing the case of an odd number of strands in this paper. There is, of course, an analogue construction. In the five strand case, for instance, we set $\gamma(5,L)=\Delta_5$ . $\sigma_1 \sigma_2 \sigma_3 \sigma_2 \sigma_1$ . $(\sigma_1 \sigma_3 \sigma_2 \sigma_1$~. $\sigma_1 \sigma_2 \sigma_1 \sigma_3 .)^\kappa$ $\sigma_1 \sigma_3 \sigma_2 \sigma_1 \sigma_4$ (with $L=3+2\kappa$). It can be shown that the sliding circuit set of $\gamma(5,L)$ has again exactly $2(L-1)L^{N-3}=2(L-1)L^2$ elements.
\end{itemize}
\end{example}

Next we have to investigate the properties of our braids $\gamma(N,L)$. 
Here is a table, containing for each value of $N$ and $L$ which we have checked, the size of the sliding circuit set $|SC(\gamma(L,C))|$, as determined by the program~\cite{JuanProgram}. We did not check for $N=8$ with larger values of~$L$, or for any larger values of~$N$, because the calculations become too long. Apart from the previously mentioned exceptional pair $(N=4, L=3)$, our calculations confirm that the sliding circuit set has $2(L-1)L^{N-3}$ elements, consisting of $L^{N-3}$ cycling-orbits.

\begin{center}
\begin{tabular}{c||c|c|c|c|c|c|c|c|c|c|c|}
$N \backslash L$  & \ 3 \ & 5 & 7 & 9 & 11 & 13 & 15 & 17\\ \hline\hline
4 & (6) & 40 & 84 & 144 & 220 & 312 & 420 & 544 \\
6 & -- & 1000 & 4116 & 11664 & 26620 & 52728 & 94500 & 157216 \\ 
8 & -- & -- & 201684 & 944784 &  &  & \\
\end{tabular}
\end{center}

While the precise formulae for these sizes may be difficult to prove formally, there is a nice argument, involving subsurfaces and ouroboroi, for a lower bound:

\begin{lemma}\label{L:SClowerbound}
$$|SC(\gamma(L,C))|\ \geqslant \ 2(L-1)\cdot {K+N-3 \choose N-3} \ = \ 2(L-1)\cdot {\frac{L-5+N}{2} \choose N-3}$$
(recalling that $L=N-1+2K$). Thus for any fixed~$N$ and as a function of~$L$, the sliding circuit set grows at least like a polynomial of degree~$N-2$.
\end{lemma}

\begin{proof}
The key observation is that many other elements of the sliding circuit set of $\gamma(N,L)$ can be constructed by hand, in the following way. Let us look at the sequence of down-up sequences in Figure~\ref{F:BadSymbolic}, which illustrates the case $N=8$. We recall that the $N$th row is a down-up sequence, and that the following two down-up sequences are obtained from this one by making all even labels switch sides, twice. The resulting two Garside-factors are repeated $K$ times.

We can modify this symbolic picture:
in a similar way, we can replace rows number 3, 4, ... , $N-1$ each with three rows, where each row is obtained from the previous by making all even labels switch sides. Also, the resulting two Garside-factors can be repeated any number of times -- we will denote the number of repeats $k_{N-2}, k_{N-3},\ldots, k_2$. 

For instance, the fourth row of labels in Figure~\ref{F:BadSymbolic} can be replaced ($\rightsquigarrow$) as follows:

\begin{tikzpicture}[scale=0.4]
\draw[line width = \Bdiag] (2,2+\by) -- (2+0.7,2+\by+0.3);
\draw[line width = \Bdiag] (3,2+\by) -- (3+0.7,2+\by+0.3);
\draw (1,2) node{8};
\draw (2,2) node{7}; \draw (2,2)+(-\bx,-\by) rectangle +(\bx,\by);
\draw (3,2) node{6}; \draw (3,2)+(-\bx,-\by) rectangle +(\bx,\by);
\draw (4,2) node{5};
\draw (5,2) node{2};
\draw (6,2) node{1};
\draw (7,2) node{3};
\draw (8,2) node{4};
\draw[line width = \Bdiag] (3,2-\by) -- (3,2-\by-0.3);
\draw[line width = \Bdiag] (4,2-\by) -- (4-0.6,2-\by-0.3);

\draw (10,2) node{$\rightsquigarrow$};

\draw[line width = \Bdiag] (13,4+\by) -- (13+0.7,4+\by+0.3);
\draw[line width = \Bdiag] (14,4+\by) -- (14+0.7,4+\by+0.3);
\draw (12,4) node{8};
\draw (13,4) node{7}; \draw (13,4)+(-\bx,-\by) rectangle +(\bx,\by);
\draw (14,4) node{6};
\draw (15,4) node{5};
\draw (16,4) node{2};
\draw (17,4) node{1};
\draw (18,4) node{3};
\draw (19,4) node{4};
\draw (20,3) node[right]{Even labels switch sides};
\draw (12,2) node{7};
\draw (13,2) node{5}; 
\draw (14,2) node{4}; 
\draw (15,2) node{1};
\draw (16,2) node{2};
\draw (17,2) node{3};
\draw (18,2) node{6};
\draw (19,2) node{8};
\draw (20,1) node[right]{Even labels switch sides};
\draw (12,0) node{8};
\draw (13,0) node{7};
\draw (14,0) node{6}; \draw (14,0)+(-\bx,-\by) rectangle +(\bx,\by);
\draw (15,0) node{5};
\draw (16,0) node{2};
\draw (17,0) node{1};
\draw (18,0) node{3};
\draw (19,0) node{4};
\draw[line width = \Bdiag] (14,0-\by) -- (14,0-\by-0.3);
\draw[line width = \Bdiag] (15,0-\by) -- (15-0.6,0-\by-0.3);

\draw[decorate, decoration=coil] (30.5,-0.2) -- node[right=0.1cm]{$k_5\times$} (30.5,4);
\end{tikzpicture}

As a further modification, the variable $K$ counting the repeats of the last two Garside factors will be replaced by $k_1$. The result of this modification, in the case $N=6$, is shown in Figure~\ref{F:BadBraid}. We will denote $\gamma(N,k_1,\ldots,k_{N-2})$ the resulting braid.

\begin{claim}
If $k_1,\ldots,k_{N-2}$ and $k'_1, \ldots, k'_{N-2}$ are non-negative integers with $k_1 + \ldots + k_{N-2} = k'_1 + \ldots + k'_{N-2}$, then the braids $\gamma(N,k_1,\ldots,k_{N-2})$ and $\gamma(N,k'_1,\ldots,k'_{N-2})$ are conjugate.
\end{claim}

Before proving the claim, we observe an immediate consequence: $SC(\gamma(N,N-1+2K))$ has at least as many elements as there are ways of writing $K$ as a sum of $N-2$ non-negative integers, i.e. 
${K+N-3 \choose N-3}$ elements. Moreover, applying the cycling operation to these elements repeatedly, we cycle through the $L-1$ non-$\Delta$ factors twice before returning to the starting point. Therefore we obtain $2(L-1){K+N-3 \choose N-3}$ elements of the sliding circuit set. This completes the proof of Lemma~\ref{L:SClowerbound}, modulo the claim.

In order to prove the claim, we observe that $\gamma(N,k_1,\ldots,k_{N-2})$ has $N-2$ ouroboroi, which are nested inside each other. This is best seen in Figure~\ref{F:BadBraid}. The innermost ouroboros contains the arcs of the braid connecting labels 1 or 2 (the red arcs in the figure), and its head-tail is the last factor of the ``body-block'' (also indicated in the figure).  The next ouroboros out contains arcs connecting labels 1, 2 and 3 (red and yellow arcs in the figure), and its head-tail is indicated in the figure. Note that the head-tail of the ``red'' ouroboros is entirely contained within the ``red-yellow'' ouroboros. And so on, through to the ``red-yellow-green-blue'' ouroboros.  

Now ``slithering ouroboroi'' corresponds to modifying the coefficients $k_1,\ldots,k_{N-2}$ while keeping their sum constant. How can we modify these coefficients using conjugations? Let's look at Figure~\ref{F:BadBraid} again. 

Conjugating $\gamma(6,k_1,k_2,k_3,k_4)$ by $\sigma_3^2$ yields $\gamma(6, k_1+1, k_2-1, k_3, k_4)$. That is, conjugating by a positive full twist involving the strands inside the innermost ouroboros increases $k_1$ by~1 and decreases $k_2$ by~1. Similarly, we observe that conjugating $\gamma(6,k_1,k_2,k_3,k_4)$ by  $(\sigma_2\sigma_3\sigma_2)^{-2}$ yields $\gamma(6,k_1, k_2+1,k_3-1,k_4)$. That is, conjugating by the \emph{negative} full twist involving the strands inside the second-slimmest ouroboros increases $k_2$ by~$1$ and decreases $k_3$ by~1. More generally, conjugating $\gamma(N,k_1,\ldots,k_{N-2})$ by the full twists involving the strands inside the $i$th-innermost ouroboros either has the effect of increasing $k_i$ by~1 and decreasing $k_{i+1}$ by~1, or the opposite effect, depending on the parity of~$i$.
This completes the proof of the claim, and hence of Lemma~\ref{L:SClowerbound}.\end{proof}


\subsection{An example with eccentric ouroboroi}\label{SS:EccentricExample}

Consider the rigid pseudo-Anosov braid on five strands, with infimum~0 and supremum~$L$ (where $L$ is an odd number with $L\geqslant 3$)
$$\delta(5,L)=\sigma_1 \sigma_3 . \sigma_1 \sigma_2 \sigma_3 \sigma_2 \sigma_1 \sigma_4 . \sigma_1 \sigma_2 \sigma_4 \sigma_3 \sigma_2 \sigma_1 (. \sigma_1 \sigma_2 \sigma_3 \sigma_2 . \sigma_2 \sigma_3 \sigma_2 \sigma_1)^{\frac{L-3}{2}}$$ 
Computer calculations show that 
$$
|SC(\delta(5,L))| \ = \ 2L^3
$$
We won't prove this fact, but we will prove that the sliding circuit set has at least cubic growth as a function of $L$.
Indeed, it contains in particular the braids
\begin{align*}
\delta_{a,b,c}:= \ & \sigma_1 \sigma_2 \sigma_3 \sigma_2 \sigma_4 . \sigma_2 \sigma_4 \sigma_3 \sigma_2 \sigma_1 . (\sigma_1 \sigma_2 \sigma_3 \sigma_2 . \sigma_2 \sigma_3 \sigma_2 \sigma_1 .)^a  \sigma_1 \sigma_3 .\\
 &  (\sigma_1 \sigma_2 \sigma_3 \sigma_2 . \sigma_2 \sigma_3 \sigma_2 \sigma_1 .)^b \sigma_1 \sigma_2 \sigma_3 \sigma_2 \sigma_1 . \sigma_1 \sigma_2 \sigma_3 \sigma_2 \sigma_1 (. \sigma_1 \sigma_2 \sigma_3 \sigma_2 . \sigma_2 \sigma_3 \sigma_2 \sigma_1)^c
\end{align*}
for all triples of integers $(a,b,c)$ with $a,b,c\geqslant 0$ and $2a+2b+2c=L-5$, and all their cyclic conjugates. This yields a lower bound of
$$
|SC(\delta(5,L))|\ \geqslant \ L\cdot{\frac{L-5}{2}+2 \choose 2} \ = \ \frac{L^3-3L+2L}{8}
$$
In order to simplify notation in the following discussion, let us rewrite
$$\delta_{a,b,c} \ = \ X \ . \ S^a \ . \ Y \ . \ S^b \ . \ Z \ . \ S^c$$
Now, one can inspect $\delta_{a,b,c}$ for the presence of ouroboroi. One finds that there are three ouroboroi which are again nested.
There is an innermost, eccentric, one, involving the third and fourth strands, with its head-tail in the factor~$Z$. The next one out involves strands number 2, 3, and 4, and it is round with its head-tail in the factor~$Y$. The outermost one is eccentric again, involves strands 1, 2, 3, and 4, and has its head-tail in the factor~$X$.

We did not manage to generalise this construction, in order to obtain a family of braids with $N$~strands and sliding circuit sets of size $2L^{N-2}$.


\subsection{Absence of ouroboroi: counterexamples for short braids}\label{SS:NoOuroboroi}

We recall the optimistic intuition underlying our Commutativity Conjecture: ``Having a large sliding circuit set is a \emph{geometric} property, it comes from slithering ouroboroi''.
In this subsection we give some examples to show that the most simple-minded interpretations of this intuition cannot be true.

\begin{example}
There is one fairly obvious way of showing the limits of the intuition: if $x$ is a rigid braid, then for any positive integer~$p$ we have $|SC(x^p)|\geqslant |SC(x)|$, because the $p$th power of any element of $SC(x)$ belongs to~$SC(x^p)$. In fact, for all the examples discussed in previous sections, $|SC(x^p)|= |SC(x)|$. Now, for $p\geqslant 2$, the braid~$x^p$ has no ouroboroi. Indeed, one naturally obtains not a snake biting its own tail, but a snake biting the tails of a second snake, and so on until the $p$th snake bites the first snake's tail again.
\end{example}

\begin{example}
Here as some examples of rigid pseudo-Anosov braids with unexpected, and more worrying, behaviour. 
In each case, we have a fairly large sliding circuit set, but no, or not enough, ouroboroi. Now, in each case, there \emph{are} some features reminiscent of ouroboroi -- different parts of the braid commuting past each other, and low translation length in the curve complex. Also, all these examples occur for fairly short braids, and in each case there is no obvious way to construct from them infinite families of arbitrarily long braids with huge sliding circuit sets. Still, these examples pose definite challenges to our Commutativity Conjecture.  (For better readability, we are suppressing the letter~$\sigma$ from our notation.)
\begin{itemize}
\item for $N=5, L=1$: $3 2 1 4 3$. 
Here $|SC|=2$. The definition of an ouroboros doesn't really make sense for such a short braid. 
\item for $N=5, L=2$: $1 4 . 1 2 3 4 3 2 1$. Here $|SC|=10$ (with 5 orbits under cycling). It is also hard to see anyting like an ouroboros.
\item for $N=5, L=3$: $\Delta . 2 4 . 2 1 3 2 4 3 2 1$.  Here $|SC|=36$ (with 9 orbits under cycling). We see one ouroboros containing four strands, but not enough in order to explain the large sliding circuit set. Moreover,  in many elements of the sliding circuit set, e.g. $\Delta . 1 2 3 2 1 4 3 . 3 2 4$, we do not see any ouroboroi. 
\item for $N=5, L=4$: $1 2 1 3 2 4 3 . 3 4 3 . 3 2 1 4 3 2 1 . 1 2 1$. Here $|SC|=100$ (with 25 orbits under cycling). There are no ouroboroi. (There are several subsurfaces with roughly ouroboros-like behaviour, but nothing that would really make us expect such a large sliding circuit set.)
\item for $N=6, L=3$: $3 2 1 4 3 5 . 1 2 3 2 1 4 3 2 5 . 2 1 3 2 4 3 5 4 3$. Here $|SC|=234$ (with 78 orbits under cycling)
This is one of the very few examples of rigid pseudo-Anosov braids which we know of where the bound $|SC|\leqslant 2L^{N-2}$ does not hold -- indeed, $234>162=2L^{N-2}$; worse, there are no ouroboroi. This example is very strange.
\item similarly, for $N=7, L=2$: $2 1 3 2 1 6 5 . 1 2 1 3 2 4 5 4 3 2 6$. Here $|SC|=92$ (with 46 orbits under cycling), and there are no ouroboroi. Notice that $92>2L^{N-2}=64$. 
\end{itemize}
\end{example}


\section{Conjectures concerning the structure of sliding circuit sets}\label{S:Conjectures}


\subsection{Polynomial bounds on the sliding circuit set}

We recall that the remaining difficulty of the conjugacy problem stems essentially from the fact that we do not know a polynomial bound for the size of the sliding circuit set.
For many rigid braids, the sliding circuit set contain only the braid itself, plus the obvious conjugates (those obtained by cyclic permutation of the factors and by conjugation by Garside's~$\Delta$). However, sometimes the sliding circuit set is much larger, as seen in the examples in the previous section. As an answer to Question~\ref{Q:SCpolynBound?} we propose:

\begin{conjecture}\label{C:deltaWorst}
The family of braids~$\gamma(N,L)$ exhibited in Section~\ref{SS:NestedOuroboroi} is the worst possible: 
there exists a constant $C$ such that for any rigid braid with $N$ strands and Garside-length~$L$, the sliding circuit set has at most $C\cdot L^{N-2}$ elements. Moreover, this bound holds with $C=2$ for 
sufficiently large values of~$L$.
\end{conjecture}

Note that, for braids with \emph{four} strands, we already know from work of Sang Jin Lee~\cite{SJLeeThesis} that even the \emph{super summit set} of four strand braids (which contains the sliding circuit set) has quadratically bounded size, even for non-pseudo-Anosov braids. (The statement that the sliding circuit set of four-strand braids, in the \emph{dual} Garside-structure, is quadratically bounded was also proved independently by Calvez and Wiest~\cite{CalvezWiestPA}.)


\subsection{Commutativity conjecture}

This is the main conjecture of this paper: large sliding circuit sets should come from some kind of internal commutativity of the braid, and this internal commutativity is a geometric fact. Our definition of an ouroboros is our attempt to capture this idea in a precise mathematical framework.

Now,
while the examples in Section~\ref{SS:NoOuroboroi} are worrying, we only found such bad behaviour for very short braids, despite checking many millions of randomly generated examples (using random walks in the Cayley graph with Artin generators). Obviously, this may just mean that the really bad examples get exceedingly rare as the length increases. Still, it is tempting to believe that for sufficiently long braids, the property of having very big sliding circuit sets is geometric:

\begin{conjecture}\label{C:BigSCisGeometric}
For every~$N$ (with $N\geqslant 5$), there exists an integer $L_N$ with the following property:
if~$x\in B_N$ is a braid which maximises $|SC(x)|$ among all rigid pseudo-Anosov braids in~$B_N$
of length~$L$, with $L\geqslant L_N$, then $x$ has $N-2$ ouroboroi, and $SC(x)$ is obtained from $x$ by sliding ouroboroi. 
\end{conjecture}

\begin{remark}
\begin{enumerate}
\item Conjecture~\ref{C:BigSCisGeometric} is still a little bit vague -- we haven't properly defined the property of $SC$ being obtained by sliding ouroboroi. Nevertheless, a proof of any reasonable interpretation of  it would imply Conjecture~\ref{C:deltaWorst}.
\item The Commutativity Conjecture is an incarnation of the general philosophy that Garside theory is a quasi-geometric theory. Another incarnation of this principle is the conjecture of Calvez-Wiest~\cite{CalvezWiestCAL} that the additional length complex of~$B_N$, equipped with its standard Garside structure, is quasi-isometric to the curve complex of the $N$-punctured disk.
\end{enumerate}
\end{remark}


\subsection{Cubing conjecture}

We finish with some comments on the \emph{uniform} conjugacy problem. 
One step in this direction might be to find an algorithm for the conjugacy problem whose computational complexity depends polynomially not only on the length~$L$ but also on the number of strands~$N$. 
The classical Garside-theoretic approach is hopeless for this purpose, as we know that the size of the sliding circuit set \emph{can} depend exponentially on~$N$.
However, calculating the full sliding circuit set may not be necessary for solving the conjugacy problem. 

In~\cite{BirGebGM2}, the Ultra Summit Set of  a braid~$x$ (which, for rigid braids with at least two non-$\Delta$ factors coincides with the sliding circuit set~\cite{Gebhardt-GM}) is considered as a graph, with vertices corresponding to elements of $USS(x)$ and edges corresponding to minimal conjugations. 
We wish to understand the overall shape of this space -- let us call it the sliding circuit set graph. 
Unfortunately, the deep results from the paper~\cite{BirGebGM2} are not sufficient for this purpose. We point out one partial result:

\begin{proposition}[Linearly bounded diameter]\label{L:LinBoundedDiam}
For any fixed number of strands~$N$, there is a constant $C_N$ such that any two elements of~$SC(x)$ (for $x\in B_N$) are related by a sequence of at most $C_N\cdot length(x)$ elements of $SC(x)$, with each element obtained from its predecessor by conjugation by a simple element.
\end{proposition}

\begin{proof}
This is an immediate consequence of the linearly bounded conjugator length in mapping class groups \cite{MasurMinsky2,Tao} and the convexity of the sliding circuit set~\cite[Proposition 7]{Gebhardt-GM}.
\end{proof}

Also, the procedure of slithering multiple ouroboroi so as to change their relative position corresponds to conjugations by mutually commuting braids. In other words, slithering ouroboroi translate to (1-skeleta of) cubes in the sliding circuit set graph. In the examples presented in Section~\ref{S:Examples}, we have many cubes gluing together to form large cubes -- the side-length of these cubes grows with~$L$, and their dimension grows with~$N$.

Our previous observations and conjectures suggest that the overall structure of the sliding circuit set is essentially that of a big cube, or a limited number of big cubes. This idea is certainly simplistic, as we know from examples presented in~\cite{Gebhardt-GM}, or from Proposition~5.3 of \cite{CalvezWiestPA} that there may be \emph{some} branching in the sliding circuit set. Nevertheless, the following vague conjecture should have some truth to it:

\begin{conjecture}[Cubing conjecture]
If the sliding circuit set is large, then large regions of the sliding circuit set graph should be cubical. 
\end{conjecture}

A very fast algorithm for solving the conjugacy problem should pick out some small preferred subset of the sliding circuit set. It could do so by holding one ouroboros in place, and slithering the other ouroboroi upwards until they are completely jammed against the fixed one, in a corner of a large cube of the sliding circuit set graph. Alternatively, one could end up in a different cubical region, or in a smaller, more interesting, region of the sliding circuit set.

%


\end{document}